\documentclass[a4paper,12pt]{amsart}

\usepackage{amsmath,amssymb,amsfonts,latexsym}

\newtheorem{theorem}{Theorem}[section]

\newtheorem{prop}[theorem]{Proposition}
\newtheorem{eg}[theorem]{Example}
\newtheorem{defn}[theorem]{Definition}
\newtheorem{remark}[theorem]{Remark}

\usepackage{enumerate}

\begin{document}

\title[Tensor frames]{Frame related sequences in tensor product of Hilbert spaces}

\author{Hemalatha M}
\address{Department of Mathematical and Computational Sciences, National Institute of Technology Karnataka, Mangaluru, India. 575025. }
\email{hemapadma28@gmail.com}


\begin{abstract}
We study frames and related sequences in tensor products of separable Hilbert spaces from an algebraic perspective. For sequences $\{f_n\}_{n\in\mathbb{N}} \subset H_1$ and $\{g_m\}_{m\in\mathbb{N}} \subset H_2$, we consider the tensor sequence $\{f_n \otimes g_m\}_{n,m\in\mathbb{N}}$ in $H_1 \otimes H_2$. We characterize the lower semi-frame and Riesz--Fischer properties of tensor sequences in terms of the corresponding properties of the component sequences. These results extend to infinite tensor products. In addition, we study the action of tensor sums of operators on frame-related sequences in tensor product Hilbert spaces.
\vspace{2mm}


\vspace{2mm}


\end{abstract}

\maketitle

 \section{Introduction}
The theory of frames in Hilbert spaces generalizes orthonormal bases and plays an important role in modern functional analysis. Frames were introduced by Duffin and Schaeffer (1952) in the study of nonharmonic Fourier series and have since found applications in signal processing, wavelet theory, sampling theory, quantum information, and operator theory. Frames provide stable and redundant representations of vectors in Hilbert spaces. Related concepts such as Bessel sequences, Riesz bases, lower semi-frames, and Riesz–Fischer sequences extend frame theory to settings where redundancy or completeness conditions are weakened. The operator-theoretic characterization of these sequences through analysis and synthesis operators has proved useful. For further background on frames and related sequences, see \cite{1,2, 2011, 2002, 7, Dsum, hema}

The study of frames in tensor products of Hilbert spaces arises naturally in multivariate harmonic analysis, quantum mechanics, and the theory of composite systems. If $H_1$ and $H_2$ are separable Hilbert spaces, for sequences $\{f_n\}_{n \in \mathbb{N}} \subset H_1$ and $\{g_m\}_{m \in \mathbb{N}} \subset H_2$, one may construct the tensor sequence $\{f_n \otimes g_m\}_{n,m \in \mathbb{N}} \subset H_1 \otimes H_2.$ Earlier work has shown that frame-related properties, such as completeness and Bessel conditions, behave well under tensor formation with appropriate assumptions (see \cite{Wiedmann, tensor dual}). 

The aim of this work is to study frame-related sequences in finite and infinite tensor products of Hilbert spaces. We establish conditions under which the lower semi-frame property, and the Riesz-Fischer property transfer from component sequences to their tensor sequences. These results are obtained using operator-theoretic methods and structural properties of tensor products. We also analyze the action of tensor products and tensor sums of densely defined linear operators on frame-related sequences and derive criteria for the preservation of frame properties.

Extending these results to infinite tensor products presents additional difficulties due to the lack of a canonical basis structure and the dependence on reference sequences. Consequently, frame and semi-frame properties may not automatically persist in infinite tensor product settings. We identify conditions ensuring their preservation and provide examples illustrating the limitations of such extensions.
\section{Preliminaries}  
This section summarizes the essential concepts used throughout the paper, including frame-related sequences,  analysis operator and basic operator- theoretical results which we used in following section. We also review the structure of tensor products of Hilbert spaces. These preliminaries establish the notation and tools required for the results that follow.
Throughout the paper, $H_1$ and $H_2$ denote separable Hilbert spaces. For $f \in H_1$ and $g \in H_2$, the tensor product of $f$ and $g$ is written as $f \otimes g$, while the tensor product of $H_1$ and $H_2$ is denoted by $H_1 \otimes H_2$.
Let $A$ and $B$ be densely defined operators on $H_1$ and $H_2$, respectively. The tensor product of $A$ and $B$ is denoted by $A \otimes B$ and the tensor sum by $A \boxplus B$. The domain, null space and range of an operator $A$ are denoted by $\mathrm{Dom}(A)$, $N(A)$ and $R(A)$ respectively. The set of natural numbers is denoted by $\mathbb{N}$.
\begin{defn} \cite{1,2002,7}
		Let $F = \{f_n\}_{n\in \mathbb{N}}$ be a sequence in a Hilbert space $H$. 
  \begin{enumerate}[(i)]
 \item  $F$ is said to be a complete sequence for $H$ if $\overline{span\{f_n\}_{n \in \mathbb{N}}} = H$.
      \item $F$ is said to be a frame for $H$ if there exist constants $0< A \leq B < \infty$ such that
		\begin{equation*}
		A\| f \| ^2 \leq \sum_{n \in \mathbb{N}} | \langle f, f_n \rangle |^2 \leq  B\| f \| ^2, \hspace*{0.5cm}  \text{for all     } f \in H.\end{equation*}
  \item $F$ is said to be a Bessel sequence for $H$ if it satisfies the right hand side of the frame inequality.
 \item $F$ is said to be an upper semi-frame for $H$ if $F$ is a complete Bessel sequence in $H.$
  \item $F$ is said to be a lower semi-frame for $H$ if $F$ is a complete sequence for $H$ and satisfies the left hand side of the frame inequality.
  $F$ is said to be a lower semi-frame sequence for $H$ if $\{f_n\}_{n \in \mathbb{N}}$ is a lower semi-frame for $\overline{span\{f_n\}_{n \in \mathbb{N}}}$.
  \item $F$ is said to be a Riesz sequence for $H$ if there exist constants $0 < A \leq B < \infty$ such that
	\begin{equation*}
	A \sum_{n \in \mathbb{N}} | c_n| ^2 \leq \| \sum_{n \in \mathbb{N}} c_n f_n \| ^2 \leq B \sum_{n \in \mathbb{N}} | c_n| ^2
	\end{equation*}
	for all finite scalar sequences $\{c_n\}_{n \in \mathbb{N}} \in \ell^2$.
 \item $F$ is said to be a Riesz-Fischer sequence for $H$ if $F$ satisfies the left hand side of the Riesz sequence inequality.
  
   \end{enumerate}
	\end{defn}  
	
Consider a sequence $\{f_n\}_{n \in \mathbb{N}}$ in $H$, we can define analysis operator as $C: Dom(C) \subseteq H \rightarrow \ell^2$ is defined by \begin{eqnarray*}
			Dom(C) &=& \{f \in H : \sum_{n=1}^{\infty} | \langle f, f_n \rangle |^2 < \infty \}\\
			Cf &=& \{\langle f, f_n \rangle\}_{n \in \mathbb{N}}, \quad \text{ for } f \in Dom(C).
		\end{eqnarray*}
 \begin{defn}\cite{kato}
     Let $A$ be a densely defined closed operator on $H$. The reduced minimum modulus for $A$ is defined by $$\gamma(A) = \inf \Big\{ \frac{\|Af\|}{\|f\|} : f \neq0 \in Dom(A) \cap N(A)^\perp\Big\}.$$
 \end{defn}
 \begin{prop} \cite{4a} \label{closed}
     Let $A$ be a densely defined closed operator on $H$. Then the following statements are equivalent:
     \begin{enumerate}[(i)]
         \item $\gamma(A) >0$.
         \item $R(A)$ is closed.
         \item $R(A^*)$ is closed.
     \end{enumerate}
 \end{prop}
 \begin{prop}\label{1} \cite{2011}
	Let $\{f_n\}_{n \in \mathbb{N}} $ be a sequence  in $H$. Then the following statements hold.
	\begin{enumerate}[(i)]
     \item $\{f_n\}_{n \in \mathbb{N}} $ is complete in $H$ if and only if $C$ is injective.
		\item   $\{f_n\}_{n \in \mathbb{N}} $ is a lower semi-frame for $H$ if and only if $C$ is injective and $R(C)$ is closed.
		\item $\{f_n\}_{n \in \mathbb{N}} $ is a Riesz-Fischer sequence if and only if $C$ is surjective.
        
	\end{enumerate}
\end{prop}

Let $H_1$ and $H_2$ be separable Hilbert spaces.  Let $H_1$ and $H_2$ be separable Hilbert spaces over $F.$ For each $(x,y) \in H_1 \times H_2$,define a map $x\odot y : H_1 \times H_2 \rightarrow F$ such that $(x \odot y)(f,g) = \langle x,f \rangle\langle y,g\rangle$. Let $$H_1 \odot H_2 = \left\{\sum _{i=1}^{n} x_i \odot y_i: (x_i,y_i )\in H_1 \times H_2 \right\}.$$ Define a sesquilinear form on $H_1 \odot H_2$ as 
$$
\langle f_1 \odot g_1, f_2 \odot g_2 \rangle 
= \langle f_1, f_2 \rangle_{H_1}\langle g_1, g_2 \rangle_{H_2},
$$
and extend it linearly on $H_1 \odot H_2$. This form is positive definite and its induced norm is 
$$
\|f \odot g\| = \|f\|_{H_1}\|g\|_{H_2}.
$$
The tensor product of Hilbert spaces $H_1$ and $H_2$, $H_1 \otimes H_2$ is defined as the completion of $H_1 \odot H_2$ with respect to this norm. For more details refer to \cite{tensor, Wiedmann}.
Let $\{H_n\}_{n \in \mathbb{N}}$ be a sequence of separable Hilbert spaces. In general, the algebraic tensor product $\odot_{n \in \mathbb{N}} H_n$ does not admit a canonical Hilbert space completion without additional structure. A standard construction, due to von Neumann \cite{infinite}, introduces a reference sequence 
$\Omega = \{\Omega_n\}_{n \in \mathbb{N}}$, where $\Omega_n \in H_n$ and $\|\Omega_n\| = 1$ for all $n \in \mathbb{N}$. 
For elementary tensors $x = \odot_{n=1}^\infty x_n$ and 
$y = \odot_{n=1}^\infty y_n$, where $x_n=y_n=\Omega_n$ for all but finitely many $n$. 
Define the sesquilinear form on $\odot_{n \in \mathbb{N}} H_n$ as
$$
\langle x, y \rangle := \prod_{n=1}^\infty \langle x_n, y_n \rangle_{H_n}.
$$
The product converges since all but finitely many terms are $1$. Completing this 
pre-Hilbert space yields $\bigotimes_{n\in\mathbb{N}} (H_n,\Omega_n)$, called the von Neumann infinite tensor product.\\
Let $A$ and $B$ be densely defined operators on $H_1$ and $H_2$, respectively. Now, define the tensor product of $A$ and $B$, $A \otimes B$ on $H_1 \otimes H_2$ by $(A \otimes B) (x \otimes y) = Ax \otimes By$ with domain of $A \otimes B$ is $Dom(A \otimes B) = \{x \otimes y \in H_1 \otimes H
_2 : Ax \otimes By \in H_1 \otimes H_2\} =Dom(A) \odot Dom(B) $ and the range $R(A \otimes B) = R(A) \odot R(B)$. Moreover, define the tensor sum of $A$ and $B$, $A \boxplus B$ on $H_1 \otimes H_2$ by $A \boxplus B = A \otimes I_{H_1} + I_{H_2} \otimes B$ with $Dom(A \boxplus B) = Dom(A) \odot Dom(B)$. 
\begin{prop}\cite{tensor}
    Let $A_i$ and $B_i$ be bounded linear operators on Hilbert spaces $H_1$ and $H_2$ respectively, where $i=1,2.$ Then, the following hold. 
    \begin{enumerate}[(i)]
        \item $\alpha A_1 \otimes B_1 = A_1 \otimes \alpha B_1,\text{ for }\alpha \in \mathbb{C}.$\\
        \item $ (A_1+A_2) \otimes B_1 = A_1 \otimes B_1 + A_2 \otimes B_1$.\\
        \item $A_1 \otimes (B_1+B_2) = A_1 \otimes B_1+ A_1 \otimes B_2$.\\
        \item $ (A_1 \otimes B_1)(A_2 \otimes B_2) = A_1A_2 \otimes B_1B_2$.\\
       \item  $(A_1 \otimes B_1)^* = A_1^* \otimes B_1^*$.\\
       \item  If $A_1$ and $B_1$ are invertible, then
        $ (A_1 \otimes B_1)^{-1} = A_1^{-1} \otimes B_1^{-1}$.\\
    \end{enumerate}
\end{prop}
\section{Frames in tensor products of Hilbert spaces}
In this section, we investigate how frame-related properties behave under tensor products of Hilbert spaces. We provide necessary and sufficient conditions ensuring that completeness, lower semi-frame, and Riesz–Fischer properties are preserved. We further examine how tensor products and tensor sums of operators act on such sequences.
Consider sequences $\{f_n\}_{n\in \mathbb{N}} \text{ and } \{g_m\}_{m \in \mathbb{N}}$ in $H_1 \text{ and } H_2$ respectively. We can define analysis operator for the sequence $\{f_n \otimes g_m\}_{m,n \in \mathbb{N}}$ as
$C: Dom(C) \subseteq H_1 \otimes H_2 \rightarrow \ell^2(\mathbb{N} \times \mathbb{N})$ $$ C(f\otimes g) = \{\langle f \otimes g, f_n \otimes g_m\rangle\}_{n,m \in \mathbb{N}} = \{\langle f,f_n\rangle\langle g,g_m\rangle\}_{n,m\in \mathbb{N}},$$ $$Dom(C) = \left\{  f\otimes g \in H_1 \otimes H_2 : \{\langle f \otimes g, f_n \otimes g_m\rangle\}_{n,m \in \mathbb{N}} \in \ell^2(\mathbb{N} \times \mathbb{N})\right\}.$$
We can observe from the definition of these operators that if $C_1$ and $C_2$ are the analysis operators for $\{f_n\}_{n \in \mathbb{N}}$ and $\{g_m\}_{m \in \mathbb{N}}$ respectively, then $C = U\circ (C_1 \otimes C_2)$, where $U: \ell^2(\mathbb{N}) \otimes \ell^2(\mathbb{N}) \rightarrow \ell^2(\mathbb{N \times N})$ defined as $U(\{a_n\}_{n \in \mathbb{N}} \otimes \{b_m\}_{m \in \mathbb{N}}) = \{a_nb_m\}_{n,m \in \mathbb{N}}$. Moreover, $U$ is a unitary operator. That is, \begin{eqnarray*}
    \|\{a_n\}_{n \in \mathbb{N}} \otimes \{b_m\}_{m \in \mathbb{N}} \|^2 &=& \sum_{n\in \mathbb{N}} \mid a_n \mid ^2 \sum_{m\in \mathbb{N}} \mid b_m \mid ^2\\ &=&\sum_{n,m\in \mathbb{N}} \mid a_n b_m \mid ^2\\ &=& \| \{a_nb_m\}_{n,m\in \mathbb{N}} \|^2. 
\end{eqnarray*}
\begin{theorem} \label{3a}\cite{Wiedmann}
  Let $\{f_n\}_{n \in \mathbb{N}} \subset H_1$ and $\{g_m\}_{m \in \mathbb{N}} \subset H_2$ be two sequences. Then $\{f_n \otimes g_m\}_{n,m \in \mathbb{N}} \subset H_1 \otimes H_2$ is a complete sequence if and only if $\{f_n\}_{n \in \mathbb{N}}$ and $\{g_m\}_{m \in \mathbb{N}}$ are complete sequences in $H_1$ and $H_2$ respectively.
  \end{theorem}
\begin{proof}
    Let $\{f_n \otimes g_m\}_{n,m \in \mathbb{N}}$ to be a complete sequence. Then the analysis operator $C$ is injective. Since $U$ is unitary, $N(C) = N(C_1 \otimes C_2).$ Let $f \in N(C_1), g \in H_2\backslash N(C_2)$. Then $f\otimes g \in N(C_1 \otimes C_2).$ Since $g \neq 0$ and $C$ is injective, $ f=0$. Therefore, $\{f_n\}_{n \in \mathbb{N}}$ is a complete sequence in $H_1.$ In a similar way, one can prove that $\{g_m\}_{m \in \mathbb{N}}$ is a complete sequence in $H_2.$\\
    \indent Conversely, if $C_1$ and $C_2$ are injective, then $C_1 \otimes C_2$ is injective. This concludes that $C$ is injective.
\end{proof}
\begin{theorem}\label{3c}
    Let $\{f_n\}_{n \in \mathbb{N}}$ and $\{g_m\}_{m \in \mathbb{N}}$ be two sequences in $H_1$ and $H_2$ respectively such that the analysis operator $C_1$ of $\{f_n\}_{n \in \mathbb{N}}$ and the analysis operator $C_2$ of $\{g_m\}_{m \in \mathbb{N}}$ are densely defined. Then $\{f_n\}_{n \in \mathbb{N}}$ and $\{g_m\}_{m \in \mathbb{N}}$ are lower semi-frames for $H_1$ and $H_2$ respectively if and only if $\{f_n \otimes g_m\}_{n,m \in \mathbb{N}}$ is a lower semi-frame for $H_1 \otimes H_2$.
\end{theorem}
\begin{proof}
    By the Theorem $\ref{3a}$, we can conclude that $C$ is injective if and only if $C_1$ and $C_2$ are injective. Now, consider $R(C)$ to be closed. $R(C) = U(R(C_1 \otimes C_2))$ and the continuity of $U$ give that $R(C_1 \otimes C_2)$ is closed on $\ell^2(\mathbb{N})\otimes \ell^2(\mathbb{N})$. Consider a sequence $\{a_n\}_{n \in \mathbb{N}} \in R(C_1)$ such that $(a_n) \rightarrow a_0$ as $n \rightarrow \infty$ Since $C_2$ is injective, there exists $x_0 \in Dom(C_2)$ such that $z_0:= C_2x_0 \neq 0$. Then $\{a_n \otimes z_0\}_{n \in \mathbb{N}} \in R(C_1 \otimes C_2)$ and it converges to $a_0 \otimes z_0$. Since $R(C_1 \otimes C_2)$ is closed, $a_0 \otimes b_0 \in R(C_1 \otimes C_2)$. This implies that there exists $b_0 \in Dom(C_1)$ such that $a_0 = C_1b_0$. Therefore $R(C_1)$ is closed. Similarly, we can prove that $R(C_2)$ is closed.\\
    \indent Conversely, if $R(C_1)$ and $R(C_2)$ are closed, then by closed graph theorem, $C_1^{-1}$ and $C_2^{-1}$ are bounded. Therefore, $C_1^{-1}\otimes C_2^{-1}$ is bounded on $R(C_1) \odot R(C_2).$ We can extend $C_1^{-1}\otimes C_2^{-1}$ by continuity to $R(C_1) \otimes R(C_2).$ Since $C_1^{-1}\otimes C_2^{-1}$ is the inverse of $C_1\otimes C_2$ and $C_1^{-1}\otimes C_2^{-1}$ is bounded, $C_1\otimes C_2$ is bounded below by $\parallel C_1^{-1}\otimes C_2^{-1} \parallel > 0.$ As $C_1 \otimes C_2$ is densely defined, then by Proposition \ref{closed} $R(C_1 \otimes C_2)$ is closed. Then $R(C)$ is closed. This completes the proof.
\end{proof}
\begin{theorem}
Let $\{f_n\}_{n \in \mathbb{N}}$ and $\{g_m\}_{m \in \mathbb{N}}$ be two sequences in $H_1$ and $H_2$ respectively such that the analysis operator $C_1$ of $\{f_n\}_{n \in \mathbb{N}}$ and the analysis operator $C_2$ of $\{g_m\}_{m \in \mathbb{N}}$ are densely defined. Then, $\{f_n\}_{n \in \mathbb{N}}$ and $\{g_m\}_{m \in \mathbb{N}}$ are lower semi-frame sequences for $H_1$ and $H_2$ respectively if and only if $\{f_n \otimes g_m\}_{n,m \in \mathbb{N}}$ is a lower semi-frame sequence for $H_1 \otimes H_2$.
\end{theorem}
\begin{proof}
Let $\{f_n\}_{n \in \mathbb{N}}$ and $\{g_m\}_{m \in \mathbb{N}}$ be lower semi-frame sequences for $H_1$ and $H_2$ with analysis operator $C_1$ and $C_2$ respectively. Then by Proposition \ref{closed} $\gamma(C_1)  > 0$ and $\gamma(C_2) > 0$. Now, 
\begin{eqnarray*}
    \gamma(C_1 \otimes C_2) &=& \inf\{\frac{\parallel C_1x \parallel \parallel C_2 y \parallel}{\parallel x \parallel \parallel  y \parallel} :x\otimes y \in Dom(C_1 \otimes C_2) \cap N(C_1 \otimes C_2)^{\perp}\}\\
&=& (\inf\{\frac{\parallel C_1x \parallel}{\parallel x \parallel } :x\in Dom(C_1 ) \cap N(C_1 )^{\perp}\}).\\&& \hspace{0.5cm}  (\inf\{\frac{\parallel C_2y \parallel}{\parallel y \parallel } :y\in Dom(C_2 ) \cap N(C_2 )^{\perp}\})\\&=& \gamma(C_1).\gamma(C_2) >0.
\end{eqnarray*} 
This implies that $R(C_1 \otimes C_2)$ is closed, and since $U$ is a unitary operator, $R(C)$ is closed.\\
\indent Conversely, if $\{f_n \otimes g_m\}_{n,m \in \mathbb{N}}$ is a lower semi-frame sequence for $H_1 \otimes H_2$, then $R(C)$ is closed in $\ell^2(\mathbb{N} \times \mathbb{N}).$ This gives that $R(C_1 \otimes C_2)$ is closed in $\ell^2(\mathbb{N}) \otimes \ell^2(\mathbb{N}).$ Therefore, $\gamma(C_1 \otimes C_2) =\gamma(C_1). \gamma(C_2) >0,$ which implies that $R(C_1)$ and $R(C_2)$ are closed. This completes the proof. 
\end{proof}

\begin{theorem} \label{4.4}
 Let $\{f_n\}_{n \in \mathbb{N}} \subset H_1$ and $\{g_m\}_{m \in \mathbb{N}} \subset H_2$ be two sequences. Then $\{f_n \otimes g_m\}_{n,m \in \mathbb{N}} \subset H_1 \otimes H_2$ is a Riesz-Fischer sequence for $H_1 \otimes H_2$ if and only if $\{f_n\}_{n \in \mathbb{N}}$ and $\{g_m\}_{m \in \mathbb{N}}$ are Riesz-Fischer sequences for $H_1$ and $H_2$ respectively.
\end{theorem}
\begin{proof}
Let $\{f_n \otimes g_m\}_{n,m \in \mathbb{N}}$ be a Riesz-Fischer sequence for $H_1 \otimes H_2.$ Since $C$ is surjective, $R(C_1 \otimes C_2)$ is closed. By the Theorem \ref{3c}, we can conclude that $R(C_1)$ and $R(C_2)$ are closed. Suppose $R(C_1) \neq \ell^2(\mathbb{N})$, then there exists $x_0 \neq 0\in \ell^2(\mathbb{N})$ such that $\langle x_0, C_1x\rangle = 0$ for all $x \in Dom(C_1)$. Now, let $x_0,y_0 \neq 0\in \ell^2(\mathbb{N})$ and $x \otimes y \in R(C)$. Then \begin{eqnarray*}
     \langle C(x \otimes y), U(x \otimes y)\rangle &=& \langle U(C_1x \otimes C_2y), U(x \otimes y)\rangle\\ &=& \langle C_1x \otimes C_2y, x \otimes y\rangle = 0.
\end{eqnarray*} This contradicts the fact that $R(C)$ is surjective. Therefore, $C_1$ is surjective. Similarly, we can prove that $C_2$ is surjective.\\
\indent Conversely, if $C_1$ and $C_2$ are surjective, then $$U(R(C_1 \otimes C_2)) = U(\ell^2(\mathbb{N}) \otimes \ell^2(\mathbb{N})) = \ell^2(\mathbb{N} \times \mathbb{N}).$$ This completes the proof.
\end{proof}
$A$ and $B$ are densely defined operators on $H_1$ and $H_2$ respectively. In the following theorems, we can study the properties of frame-related sequences under $A \otimes B$ and $A \boxplus B$ operators respectively.
\begin{theorem}
	Let $ \{f_{n}\}_{n \in \mathbb{N}}$ $\subset H_1, $ and  $ \{g_{m}\}_{m \in \mathbb{N}}$ $\subset H_2$.
    Let A and B be densely defined on $H_1$ and $H_2$ such that $R(A)$ and $R(B)$ are dense in $H_1$ and $H_2$, respectively. Then, $ \{f_{n}\}_{n \in \mathbb{N}}$ and $ \{g_{m}\}_{m \in \mathbb{N}}$  are complete sequences for $H_1$ and $H_2$, respectively if and only if $\{Af_n\otimes Bg_m\}_{n,m \in \mathbb{N}}$ is a complete sequence for $H_1 \otimes H_2. $
\end{theorem}
\begin{proof}
    Let $ \{f_{n}\}_{n \in \mathbb{N}}$ and $ \{g_{m}\}_{m \in \mathbb{N}}$ be complete sequences for $H_1$ and $H_2$ such that $C_1$ and $C_2$ as its analysis operators. Then $C_1$ and $C_2$ are injective. Now, the analysis operator for $ \{Af_{n} \otimes Bg_m\}_{n,m \in \mathbb{N}}$ is, $C(f\otimes g) = \{\langle f, Af_n \rangle \langle g,Bg_m\rangle \}_{n,m \in \mathbb{N}} = U(C_1A^* \otimes C_2 B^*)(f \otimes g) = U[(C_1 \otimes C_2) \otimes (A^* \otimes B^*)](f \otimes g) = U[(C_1 \otimes C_2) \otimes (A \otimes B)^*](f \otimes g)$. By our assumption, $A^*$ and $B^*$ are injective. Since $C_1,C_2, A^*, B^*$ are injective, $(C_1 \otimes C_2)\otimes(A \otimes B)^*$ is injective. Therefore,  $\{Af_n\otimes Bg_m\}_{n,m \in \mathbb{N}}$ is a complete sequence for $H_1 \otimes H_2.$\\
    Conversely, if  $\{Af_n\otimes Bg_m\}_{n,m \in \mathbb{N}}$ is a complete sequence for $H_1 \otimes H_2$, then $(C_1 \otimes C_2)\otimes(A \otimes B)^*$ is injective. Then, $C_1 \otimes C_2$ and $A^* \otimes B^*$ are injective. This implies that $C_1$ and $C_2$ are injective. Therefore $ \{f_{n}\}_{n \in \mathbb{N}}$ and $ \{g_{m}\}_{m \in \mathbb{N}}$ are complete sequences for $H_1$ and $H_2$.
\end{proof}
\begin{theorem}
 Let $ \{f_{n}\}_{n \in \mathbb{N}}$ $\subset H_1, $ and $ \{g_{m}\}_{m \in \mathbb{N}}$ $\subset H_2$ such that the analysis operators $C_1$ and $C_2$ respectively, which are densely defined on $H_1$ and $H_2,$ respectively. Let $A$ and $B $ be densely defined on $H_1$ and $H_2$ such that $R(A)$ and $R(B)$ are dense in $H_1$ and $H_2$, respectively. Assume that $\gamma(A^*) , \gamma(B^*)>0$ and $R(A^* \otimes B^*)$ is dense in $H_1\otimes H_2$. Then, $ \{f_{n}\}_{n \in \mathbb{N}}$ and $ \{g_{m}\}_{m \in \mathbb{N}}$  are lower semi-frames for $H_1$ and $H_2$, respectively if and only if $\{Af_n\otimes Bg_m\}_{n,m \in \mathbb{N}}$ is a lower semi-frame for $H_1 \otimes H_2.$
\end{theorem}
\begin{proof}
    Let $ \{f_{n}\}_{n \in \mathbb{N}}$ and $ \{g_{m}\}_{m \in \mathbb{N}}$ be lower semi-frames for $H_1$ and $H_2$. Then, $\gamma(C_1),\gamma(C_2)>0$. This gives $\gamma(C_1 \otimes C_2)>0$. Since, $\gamma(A^*) , \gamma(B^*)>0$, $\gamma(A^* \otimes B^*) >0.$ Therefore, $\gamma((C_1 \otimes C_2) \otimes (A \otimes B)) >0$ By Proposition \ref{closed}, $R(C_1 \otimes C_2) \otimes (A \otimes B))$ is closed. Therefore, $R(C)$ is closed. According to our assumptions, $C_1, C_2, A^*, \text{ and } B^*$ are injective, which guarantee that $\{Af_n\otimes Bg_m\}_{n,m \in \mathbb{N}}$ is a complete sequence on $H_1\otimes H_2.$\\ \indent Conversely, if $\{Af_n\otimes Bg_m\}_{n,m \in \mathbb{N}}$ is a lower semi-frame for $H_1 \otimes H_2$, then the analysis operator $C = U[(C_1 \otimes C_2)\otimes(A^* \otimes B^*)]$ is injective and $ \gamma((C_1 \otimes C_2)\otimes(A^* \otimes B^*))>0$. This gives that, $\gamma(C_1)>0$ and $\gamma(C_2)>0$. As $C$ is injective, we can conclude that $C_1$ and $C_2$ are injective. This completes the proof.     
\end{proof}
\begin{theorem}
  Let $ \{f_{n}\}_{n \in \mathbb{N}}$ $\subset H_1, $ and $ \{g_{m}\}_{m \in \mathbb{N}}$ $\subset H_2$ be such that the analysis operators $C_1$ and $C_2$ respectively are densely defined on $H_1$ and $H_2,$ respectively. Let A and B be densely defined operators on $H_1$ and $H_2$, respectively, such that  $R(A^*)=H_1$ and $R(B^*)=H_2$. Then, $ \{f_{n}\}_{n \in \mathbb{N}}$ and $ \{g_{m}\}_{m \in \mathbb{N}}$  are Riesz-Fischer sequences for $H_1$ and $H_2$, respectively if and only if $\{Af_n\otimes Bg_m\}_{n,m \in \mathbb{N}}$ is a Riesz-Fischer sequence for $H_1 \otimes H_2.$
\end{theorem}
\begin{proof}
 Assume that $ \{f_{n}\}_{n \in \mathbb{N}}$ and $ \{g_{m}\}_{m \in \mathbb{N}}$  are Riesz-Fischer sequences for $H_1$ and $H_2$, respectively. Then $R(C_1)=R(C_2)= \ell^2(\mathbb{N})$. Since $A^*$ and $B^*$ are surjective, $A^* \otimes B^*$ is surjective. Therefore, $R((C_1 \otimes C_2) \otimes (A \otimes B)) = \ell^2(\mathbb{N})\otimes \ell^2(\mathbb{N})$. This concludes that $\{Af_n\otimes Bg_m\}_{n,m \in \mathbb{N}}$ is a Riesz-Fischer sequence for $H_1 \otimes H_2.$\\ \indent Conversely, $\{Af_n\otimes Bg_m\}_{n,m \in \mathbb{N}}$ is a Riesz-Fischer sequence for $H_1 \otimes H_2.$ Then by Theorem \ref{4.4}, $\{Af_n\}_{n \in \mathbb{N}}$ and $\{Bg_m\}_{m \in \mathbb{N}}$ are Riesz-Fischer sequences for $H_1$ and $H_2.$ As $R(C_1A^*) = R(C_2B^*) = \ell^2(\mathbb{N}),$ we can conclude that $C_1$ and $C_2$ are surjective. This completes the proof.  
\end{proof}
\begin{theorem} \label{4.8}
 Let $ \{f_{n}\}_{n \in \mathbb{N}}$ $\subset H_1, $ and $ \{g_{m}\}_{m \in \mathbb{N}}$ $\subset H_2$ be such that the analysis operators $C_1$ and $C_2$ respectively are densely defined on $H_1$ and $H_2,$ respectively. Let $A$ be a densely defined on $H_1$ and $B$ be a bounded operator on $H_2$ such that $R(A)$ and $R(B)$ are dense in $H_1$ and $H_2$, respectively. Assume that $\gamma(A^*) > \parallel B \parallel >0$ and $R(A \boxplus B)$ is dense in $H_1\otimes H_2$. If $ \{f_{n}\}_{n \in \mathbb{N}}$ and $ \{g_{m}\}_{m \in \mathbb{N}}$  are lower semi-frames for $H_1$ and $H_2$, respectively, then $\{A\boxplus B(f_n\otimes g_m)\}_{n,m \in \mathbb{N}}$ is a lower semi-frame for $H_1 \otimes H_2.$
\end{theorem}
\begin{proof}
    Assume that $ \{f_{n}\}_{n \in \mathbb{N}}$ and $ \{g_{m}\}_{m \in \mathbb{N}}$  are lower semi-frames for $H_1$ and $H_2$, respectively. Then $C_1$ and $C_2$ are injective and bounded below. The analysis operator for the sequence $\{A\boxplus B(f_n\otimes g_m)\}_{n,m \in \mathbb{N}}$ is, \begin{eqnarray*}
        C(f \otimes g) &=&U[ (C_1 \otimes C_2) \otimes (A^* \otimes I_{H_1} + I_{H_2} \otimes B^*)]\\ &=& U[(C_1 \otimes C_2) \otimes (A^*  \boxplus B^*))]\\ &=&U[(C_1 \otimes C_2) \otimes (A  \boxplus B)^*)] .
    \end{eqnarray*}
    From our assumption, $(A \boxplus B) ^*$ is injective. Therefore, $C$ is injective. Since $C_1$ and $C_2$ are injective and bounded below, $C_1 \otimes C_2$ is so.  Now, let $f \otimes g \in Dom(C)$
    \begin{eqnarray*}
        \parallel C(f \otimes g) \parallel &=& \parallel U[ (C_1 \otimes C_2) \otimes (A^* \otimes I_{H_1} + I_{H_2} \otimes B^*)](f \otimes g) \parallel\\ & \geq & \gamma (C_1 \otimes C_2) \parallel (A^* \otimes I_{H_1} + I_{H_2} \otimes B^*) (f \otimes g) \parallel \\ &\geq &  \gamma (C_1 \otimes C_2) |\parallel A^* \otimes I_{H_1} (f \otimes g) \parallel - \parallel I_{H_2} \otimes B^*(f \otimes g)  \parallel |\\ &\geq & \gamma (C_1 \otimes C_2) (\parallel A^* \otimes I_{H_1} (f \otimes g) \parallel - \parallel I_{H_2} \otimes B^*(f \otimes g) \parallel ) \\
        &\geq &\gamma (C_1 \otimes C_2) ( \gamma(A^*) \parallel  f \otimes g \parallel - \parallel B \parallel \parallel f \otimes g  \parallel ) \\
        &\geq& \gamma (C_1 \otimes C_2) ( \gamma(A^*) - \parallel B \parallel)\parallel f \otimes g  \parallel.
    \end{eqnarray*}
    Therefore, $C$ is bounded below. Then, by Proposition \ref{closed}, $R(C)$ is closed. Therefore, $\{A\boxplus B(f_n\otimes g_m)\}_{n,m \in \mathbb{N}}$ is a lower semi-frame for $H_1 \otimes H_2.$
\end{proof}
In the following example, we show that if $B$ is a densely defined closed operator with closed range on $H_2$, $\gamma(A^*)>\gamma(B^*)>0$ and $R(A\boxplus B)$ is dense, then $\{A\boxplus B(f_n\otimes g_m)\}_{n,m \in \mathbb{N}}$ is need not be a lower semi-frame for $H_1 \otimes H_2.$
\begin{eg}
Let $ H_1=H_2=\ell^2(\mathbb N), $ and let $\{e_n\}_{n\in\mathbb N}$ denote the canonical orthonormal basis of $\ell^2(\mathbb N)$. Define $f_n=e_n,  g_m=e_m,  n,m\in\mathbb N. $ Since $\{e_n\}_{n\in\mathbb N}$ is an orthonormal basis, both
$ \{f_n\}_{n\in\mathbb N}  \text{and}  \{g_m\}_{m\in\mathbb N} $ are Parseval frames, and hence lower semi-frames.
Now define bounded diagonal operators $A,B:\ell^2(\mathbb N)\to \ell^2(\mathbb N)$ by 
$$Ae_n = \left(1+\frac{1}{n}\right)e_n, \text{ and } Be_n = -b_ne_n,$$ 
where $ b_1=\frac12,
b_n=1+\frac1n+\frac1{n^2}, n\ge2.$ Since $A$ and $B$ are bounded operators, they are densely defined and closed. Moreover, $ A=A^*,  B=B^* $ and $\gamma(A^*) = \inf_{n\in\mathbb N}\left|1+\frac1n\right|  =1 \text{ and } \gamma(B^*) =\inf_{n\in\mathbb N}|b_n| =
\frac12.$ Hence $ \gamma(A^*)>\gamma(B^*)>0.$ For $n,m\in\mathbb N$, 
\begin{align*}
A \boxplus B(e_n\otimes e_m)
&=
Ae_n\otimes e_m+e_n\otimes Be_m
\\
&=
\left(1+\frac1n\right)e_n\otimes e_m
-b_m(e_n\otimes e_m)
\\
&=
\left(1+\frac1n-b_m\right)(e_n\otimes e_m).
\end{align*}
Thus $\{e_n\otimes e_m\}_{n,m\in\mathbb N}$ is an orthonormal basis consisting of eigenvectors of $A \boxplus B$, with eigenvalues $ \lambda_{n,m} = 1+\frac1n-b_m.$ In particular, for $m=n\ge2$, 
$\lambda_{n,n} =
1+\frac1n-\left(1+\frac1n+\frac1{n^2}\right) = -\frac1{n^2}.$ Hence $|\lambda_{n,n}|=\frac1{n^2}\rightarrow 0. $ On the other hand, $ \lambda_{n,m}\neq0  \text{ forall }n,m\in\mathbb N, $ and therefore $A\boxplus B$ is injective. Since $(A \boxplus B)=(A \boxplus B)^*$, $ N((A \boxplus B)^*)=\{0\}. $ Consequently,
$\overline{R(A \boxplus B)} = N((A \boxplus B)^*)^\perp =
\ell^2(\mathbb N)\otimes \ell^2(\mathbb N),$ that is, $ R(A\boxplus B) $ is dense. Consider the sequence $ \{A \boxplus B(e_n\otimes e_m)\}_{n,m\in\mathbb N}. $ Since $\{e_n\otimes e_m\}_{n,m}$ is an orthonormal basis, $ A \boxplus B(e_n\otimes e_m) = \lambda_{n,m}(e_n\otimes e_m).$ A diagonal sequence of this form is a lower semi-frame if and only if $ \inf_{n,m}|\lambda_{n,m}|>0. $ However, $ \inf_{n,m}|\lambda_{n,m}| = 0, $ because $ |\lambda_{n,n}|=\frac1{n^2}\rightarrow 0. $ Hence $ \{(A\boxplus B)(e_n\otimes e_m)\}_{n,m\in\mathbb N} $ is not a lower semi-frame for $ \ell^2(\mathbb N)\otimes \ell^2(\mathbb N).$
\end{eg}

\begin{theorem}
  Let $ \{f_{n}\}_{n \in \mathbb{N}}$ $\subset H_1, $ and $ \{g_{m}\}_{m \in \mathbb{N}}$ $\subset H_2$ be such that the analysis operators $C_1$ and $C_2$ respectively. Let $A$ be a densely defined on $H_1$ and $B$ be a bounded operator on $H_2$, such that  $R(A^*)=H_1$, $R(B^*)=H_2$ and $\gamma(A^*) >\parallel B \parallel >0$. Then, $ \{f_{n}\}_{n \in \mathbb{N}}$ and $ \{g_{m}\}_{m \in \mathbb{N}}$  are Riesz-Fischer sequences for $H_1$ and $H_2$, respectively if and only if $\{A\boxplus B(f_n\otimes g_m)\}_{n,m \in \mathbb{N}}$ is a Riesz-Fischer sequence for $H_1 \otimes H_2.$
\end{theorem}
\begin{proof}
    Let $ \{f_{n}\}_{n \in \mathbb{N}}$ and $ \{g_{m}\}_{m \in \mathbb{N}}$  be Riesz-Fischer sequences for $H_1$ and $H_2$, respectively. Then $C_1$ and $C_2$ are surjective. Therefore, $C_1 \otimes C_2$ is also surjective. Since $A^*, B^*$  are surjective, $A^* \otimes I_{H_1}, I_{H_2} \otimes B^*$ are so. From our assumption $\gamma(A^*) >\parallel B \parallel>0$, we can conclude that $A \boxplus B$ is injective. Therefore, $R(A^* \boxplus B^*)$ is dense and as in the proof of Theorem \ref{4.8}, we can conclude that $R(A^* \boxplus B^*)$ is closed, which implies that $R(A^* \boxplus B^*)= H_1 \otimes H_2.$ Therefore, $C$ is surjective.\\
    \indent Conversely, if $\{A\boxplus B(f_n\otimes g_m)\}_{n,m \in \mathbb{N}}$ is a Riesz-Fischer sequence for $H_1 \otimes H_2.$ Then, $C$ is surjective. This gives that $C_1 \otimes C_2$ is surjective. Then by Theorem \ref{4.4}, $C_1$ and $C_2$ are surjective. Therefore, $ \{f_{n}\}_{n \in \mathbb{N}}$ and $ \{g_{m}\}_{m \in \mathbb{N}}$  are Riesz-Fischer sequences for $H_1$ and $H_2$. 
\end{proof}
\section{Frame related sequences in infinite tensor product of Hilbert spaces}
This section extends the study to infinite tensor products in the sense of von Neumann. We identify conditions under which frame-related properties continue to hold in this setting. Examples illustrate situations where preservation fails, highlighting the need for additional structural assumptions.
Let $\{\otimes _{n \in \mathbb{N}} f_{nm} \}_{m \in \mathbb{N}}$ be a sequence in $(\otimes _{n \in \mathbb{N}}H_n, \Omega_n)$. The analysis operator for this sequence is $C: Dom(C) \subseteq \otimes _{n \in \mathbb{N}}H_n \rightarrow \ell^2(\mathbb{N})$ defined by $$C(\otimes _{n \in \mathbb{N}}x_n) = \{\langle \otimes _{n \in \mathbb{N}}x_n, \otimes _{n \in \mathbb{N}}f_{nm}\rangle\}_{m \in \mathbb{N}} = \left\{\prod_{n\in \mathbb{N}} \langle x_n, f_{nm} \rangle \right\}_{m \in \mathbb{N}}.$$
\begin{theorem}
    Let $\{\otimes _{n \in \mathbb{N}} f_{nm} \}_{m \in \mathbb{N}}$ be a sequence in $(\otimes _{n \in \mathbb{N}}H_n, \Omega_n)$. $\{\otimes _{n \in \mathbb{N}} f_{nm} \}_{m \in \mathbb{N}}$ is a complete sequence if and only if $\{f_{nm}\}_{m \in \mathbb{N}}$ is a complete sequence for all $n \in \mathbb{N}.$
\end{theorem}
\begin{proof}
    Let $\{\otimes _{n \in \mathbb{N}} f_{nm} \}_{m \in \mathbb{N}}$ be a complete sequence in $(\otimes _{n \in \mathbb{N}}H_n, \Omega_n)$. Suppose that $\{f_{nm}\}_{m \in \mathbb{N}}$ is not complete for some $n \in \mathbb{N}$. Without loss of generality, consider that $\{f_{1m}\}_{_{m \in \mathbb{N}}}$ is not complete in $H_1$ and the analysis operator of $\{f_{1m}\}_{m \in \mathbb{N}}$ is $C_1.$ Then $C_1$ is not injective. Let $x_1 \in N(C_1)$, then$ \{\langle x_1, f_{1m}\rangle\}_{m \in \mathbb{N}} = 0$ and $\{\prod_{n\in \mathbb{N}}\langle x_1 \otimes_{n \in \mathbb{N} \backslash \{1\}} u_n, f_{nm} \rangle\}_{m \in \mathbb{N}} = 0$, which contradicts our assumption. \\
     \indent Conversely, assume that $\{f_{nm}\}_{m \in \mathbb{N}}$ is complete for each $n \in \mathbb{N}$. Let $\otimes_{n \in \mathbb{N}} x_n  \neq 0 \in \otimes_{n \in \mathbb{N}} (H_n, \Omega _n) $, since $\{f_{nm}\}_{m \in \mathbb{N}}$ is complete for each $n \in \mathbb{N}$, there exists $f_{N_xm} \in H_N$ such that $\langle x_N, f_{N_xm}\rangle \neq 0$. Then $\prod _{n=1}^{\infty} \langle x_n , f_{N_xm} \rangle \neq 0.$ Therefore, $\{ \langle \otimes_{n \in \mathbb{N} }x_n, \otimes f_{nm}\rangle\}_{r, m \in \mathbb{N}} \neq 0.$ This implies that $\{\otimes_{n \in \mathbb{N}} f_{nm}\}_{r, m \in \mathbb{N}}$ is a complete sequence in $\otimes_{n \in \mathbb{N}} H_n$.
\end{proof}
\begin{theorem}\label{5.2}
     Let $\{\otimes _{n \in \mathbb{N}} f_{nm} \}_{m \in \mathbb{N}}$ be a sequence in $(\otimes _{n \in \mathbb{N}}H_n, \Omega_n)$ such that $f_{nm} \neq 0$ for all $m \in \mathbb{N}$. If $\{\otimes _{n \in \mathbb{N}} f_{nm} \}_{m \in \mathbb{N}}$ is a Bessel sequence, then $\{f_{nm}\}_{m \in \mathbb{N}}$ is a Bessel sequence for all $n \in \mathbb{N}.$
\end{theorem}
\begin{proof}
   Assume that $\{\otimes _{n \in \mathbb{N}} f_{nm} \}_{m \in \mathbb{N}}$ is a Bessel sequence in $(\otimes _{n \in \mathbb{N}}H_n, \Omega_n)$ with Bessel bound $B$. 
    Let $\Omega_n$ be a unit vector in $H_n$. Let $x \in H_1$, then $x\otimes _{n\in \mathbb{N} \backslash\{1\}} \Omega _n \in  \otimes _{n \in \mathbb{N}}H_n$. 
    $$ \sum_{_{m \in \mathbb{N}}}|\langle x_1,f_{1m}\rangle|^2 = \sum_{_{m \in \mathbb{N}}}|\langle x_1\otimes _{n\in \mathbb{N} \backslash\{1\}}\Omega _n, f_{1m}\otimes _{n\in \mathbb{N} \backslash\{1\}}\Omega _n \rangle|^2 \leq B \| x_1\|^2. $$
    Therefore $\{f_{1m}\}_{m \in \mathbb{N} }$ is a Bessel sequence for $H_1$. Similarly, we can extend our proof for all $n \in \mathbb{N}$. This completes the proof.    
\end{proof}
The converse of Theorem \ref{5.2} need not be true. Consider the following example.
\begin{eg}
    Let $H_n=\mathbb{C}$ over $\mathbb{R}$ for all $n \in \mathbb{N}$ and $\Omega=(1,0)\otimes (1,0)\otimes ...$. Therefore, $(\otimes_{n\in \mathbb{N}} H_n, \Omega_n)$ is a Hilbert space. 
Let $F_n=\{(2,0),(0,2)\}$ for all $ n\in \mathbb{N},$ 
then $F_n$ is a Bessel sequence for $H_n$ with upper bound 4.
But $\otimes_{n\in \mathbb{N}} F_n$ is not a Bessel sequence in $(\otimes _{n \in \mathbb{N}}H_n, \Omega).$
Suppose, if $\otimes_{n\in \mathbb{N}} F_n$ is a Bessel sequence with bound B,
then there exists $N\in \mathbb{N}$ such that 
$B=4^N.$
We take an element in  $\otimes_{n\in \mathbb{N}} H_n$ such that 
$\otimes_{i=1}^{N+1} (0,2) \otimes_{i=N+2}^{\infty} (1,0)$ then
$\|\otimes_{n\in \mathbb{N}}x_n\|^2 = 4^{N+1}$, which is a contradiction.
\end{eg}
\begin{remark} 
If $\{f_{nm}\}_{m \in \mathbb{N}}$ is a Bessel sequence with Bessel bound $\beta_n = 1$ for all $ n\in \mathbb{N},$ then $\{\otimes _{n \in \mathbb{N}} f_{nm} \}_{m \in \mathbb{N}}$ is a Bessel sequence for $(\otimes _{n \in \mathbb{N}}H_n, \Omega_n).$
\end{remark}

\begin{theorem}\label{5.4}
Let $\{\otimes _{n \in \mathbb{N}} f_{nm} \}_{m \in \mathbb{N}} \subset H_n. $ If $\{\otimes _{n \in \mathbb{N}} f_{nm} \}_{m \in \mathbb{N}}$ is a lower semi-frame for $(\otimes _{n \in \mathbb{N}}H_n, \Omega_n)$, then $\{f_{nm}\}_{m \in \mathbb{N}}$ is a lower semi-frame for $ H_n$ for all $ n\in \mathbb{N}. $
\end{theorem}
\begin{proof}
    Assume that $\{\otimes _{n \in \mathbb{N}} f_{nm} \}_{m \in \mathbb{N}}$ is a lower semi-frame with bound A. Let $x_1 \in H_n. $ Then  $$ \sum_{m \in \mathbb{N}}|\langle x_1,f_{1m}\rangle|^2 = \sum_{m \in \mathbb{N}}|\langle x_1\otimes _{n\in \mathbb{N} \backslash\{1\}}\Omega _n, f_{1m}\otimes _{n\in \mathbb{N} \backslash\{1\}}\Omega _n \rangle|^2 \geq A \| x_1\|^2 $$ which implies $\{f_{1m}\}_{_{m \in \mathbb{N}}}$ is a lower semi-frame for $H_1. $
    Similarly, we can prove it for all ${n \in \mathbb{N}}.$
\end{proof}
The converse of the Theorem \ref{5.4} need not be true. Consider the following example.
\begin{eg}
    Let $H_n=\mathbb{C}$ over $\mathbb{R}$ for all $n \in \mathbb{N}$ and $\Omega=(1,0)\otimes (1,0)\otimes ...$. Therefore, $(\otimes_{n\in \mathbb{N}} H_n, \Omega_n)$ is a Hilbert space 
Let $F_n=\{(\frac{1}{2},0),(0,\frac{1}{2})\}$ for all $ n\in \mathbb{N}.$ 
Then $F_n$ is a lower semi-frame with lower bound $\frac{1}{4}.$  
But $\otimes_{n\in \mathbb{N}} F_n$ is not a lower semi-frame in $(\otimes _{n \in \mathbb{N}}H_n, \Omega_n).$
If $\otimes_{n\in \mathbb{N}} F_n$ is lower semi-frame with bound A,
then there exists $n\in \mathbb{N}$ such that 
$A=\frac{1}{4^n}.$
We take an element in  $\otimes_{n\in \mathbb{N}} H_n$ such that 
$\otimes_{i=1}^{n+1} (0,\frac{1}{2}) \otimes_{i=n+1}^{\infty} (1,0),$ then
$\|\otimes_{n\in \mathbb{N}}x_n\|^2 = \frac{1}{4^{n+1}}$, which is a contradiction.
\end{eg}
\begin{remark} 
Let $\{f_{nm}\}_{m \in \mathbb{N}}$ be a sequence in $H_n$ for all $n \in \mathbb{N}.$ If $\{f_{nm}\}_{m \in \mathbb{N}}$ is a lower semi-frame with lower bound $\alpha_n \geq 1$ for all $ n\in \mathbb{N},$ then $\{\otimes _{n \in \mathbb{N}} f_{nm} \}_{m \in \mathbb{N}}$ is a lower semi-frame for $(\otimes _{n \in \mathbb{N}}H_n, \Omega_n).$

\end{remark}

\begin{theorem}\label{5.6}
Let $\{f_{nm}\}_{m \in \mathbb{N}}$ be a sequence in $H_n$ for all $n \in \mathbb{N}.$ If $\{\otimes _{n \in \mathbb{N}} f_{nm} \}_{m \in \mathbb{N}}$ is a frame for  $(\otimes _{n \in \mathbb{N}}H_n, \Omega_n)$, then $\{f_{nm}\}_{m \in \mathbb{N}}$ is a frame for $H_n$ for all ${n \in \mathbb{N}}.$
\end{theorem}
\begin{proof}
    The proof of this theorem follows from Theorem \ref{5.2} and Theorem \ref{5.4}.
\end{proof}
\begin{remark}
For each $ n \in \mathbb{N}$, $\{f_{nm}\}_{m \in \mathbb{N}}$ is a Parseval frame for $H_n$ if and only if  $\{\otimes _{n \in \mathbb{N}} f_{nm} \}_{m \in \mathbb{N}}$ is a Parseval frame for  $(\otimes _{n \in \mathbb{N}}H_n, \Omega_n)$. 
\end{remark}
\begin{theorem}
    Let $\{f_{nm}\}_{m \in \mathbb{N}}$ be a sequence in $H_n$ for all $n \in \mathbb{N}.$ If $\{\otimes _{n \in \mathbb{N}} f_{nm} \}_{m \in \mathbb{N}}$ is a Riesz-basis for  $(\otimes _{n \in \mathbb{N}}H_n, \Omega_n)$, then $\{f_{nm}\}_{m \in \mathbb{N}}$ is a Riesz basis for $H_n$ for all ${n \in \mathbb{N}}.$
\end{theorem}
\begin{proof}
    Assume that $\{\otimes _{n \in \mathbb{N}} f_{nm} \}_{m \in \mathbb{N}}$ is a Riesz-basis for  $(\otimes _{n \in \mathbb{N}}H_n, \Omega_n)$, Then there exist $A$ and $B$ such that,
    $$ A\sum _{_{m \in \mathbb{N}}}\mid c_{m}\mid^2 \leq \sum _{_{m \in \mathbb{N}}}\mid c_{m} \otimes_{n \in \mathbb{N}}f_{nm}\mid^2 \leq B\sum _{_{m \in \mathbb{N}}}\mid c_{m}\mid^2,$$ for all finite sequence $\{c_{m}\}_{{m} \in \mathbb{N}} \in \ell^2(\mathbb{N)}$. Now, fix a $k \in \mathbb{N}$,  
    $$ A\sum _{_{m \in \mathbb{N}}}\mid c_{m}\mid^2 \leq\mid \sum _{k_m \in \mathbb{N}} c_{k_m}f_{km} \mid^2 = \mid \sum _{m \in \mathbb{N}} c_{k_m}f_{km} \otimes _{n \neq k \in \mathbb{N}} \Omega_n\mid^2 
    \leq  B\sum _{_{m \in \mathbb{N}}}\mid c_{m}\mid^2.$$ Therefore $\{f_{km}\}_{k,m \in \mathbb{N}}$ is a Riesz Basis for $H_k$. Similarly, we can prove for all $k \in \mathbb{N}.$ 
\end{proof}
\begin{theorem}
    Let $\{f_{nm}\}_{m \in \mathbb{N}}$ be a sequence in $H_n$ for all $n \in \mathbb{N}.$ If $\{\otimes _{n \in \mathbb{N}} f_{nm} \}_{m \in \mathbb{N}}$ is a Riesz Fischer sequence for  $(\otimes _{n \in \mathbb{N}}H_n, \Omega_n)$, then $\{f_{nm}\}_{m \in \mathbb{N}}$ is a Riesz Fischer sequence for $H_n$ for all ${n \in \mathbb{N}}.$
\end{theorem}
\begin{proof}
    Assume that $\{\otimes _{n \in \mathbb{N}} f_{nm} \}_{m \in \mathbb{N}}$ is a Riesz Fischer sequence for  $(\otimes _{n \in \mathbb{N}}H_n, \Omega_n)$, Then there exists $A$ such that,
    $$ A\sum _{_{m \in \mathbb{N}}}\mid c_{m}\mid^2 \leq \sum _{_{m \in \mathbb{N}}}\mid c_{m} \otimes_{n \in \mathbb{N}}f_{nm}\mid^2,$$ for all finite sequences $\{c_{m}\}_{{m} \in \mathbb{N}} \in \ell^2(\mathbb{N)}$. Now, fix a $k \in \mathbb{N}$,
    $$ A\sum _{_{m \in \mathbb{N}}}\mid c_{m}\mid^2 \leq\mid \sum _{m \in \mathbb{N}} c_{k_m}f_{km} \mid^2 = \mid \sum _{m \in \mathbb{N}} c_{m}f_{km} \otimes _{n \neq k \in \mathbb{N}} \Omega_n\mid^2 
    .$$ Therefore $\{f_{km}\}_{m \in \mathbb{N}}$ is a Riesz Fischer sequence for $H_k$. Similarly, we can prove for all $k \in \mathbb{N}.$ 
\end{proof}
\begin{center}
    \textbf{CONCLUSION}
\end{center}
We establish a unified operator-theoretic characterization of frame-related sequences in tensor products of separable Hilbert spaces, providing necessary and sufficient conditions for the preservation of completeness, Bessel, lower semi-frame, and Riesz–Fischer properties. These results show that tensor sequences inherit such properties precisely when the corresponding component sequences do, and we further derive preservation criteria under tensor products and tensor sums of densely defined operators. For von Neumann infinite tensor products, we identify situations where these properties persist and exhibit counterexamples demonstrating that additional uniformity assumptions are essential. The findings extend classical tensor frame theory and lay groundwork for future studies of generalized frame structures in composite Hilbert spaces.
\begin{center}
	\textbf{Acknowledgements}
\end{center}

\noindent The author would like to sincerely thank my guide Prof. P. Sam Johnson for his valuable guidance and constructive comments, which greatly helped improve the quality of this paper. The author gratefully acknowledges the financial support from the National Institute of Technology Karnataka (NITK), Surathkal.
\addcontentsline{toc}{section}{References}

\begin{thebibliography}{10}
	
	
	
	\bibitem{1}
	Antoine Jean-Pierre and Balazs Peter.
	\newblock Frames and semi-frames.
	\newblock {\em Journal of Physics A: Mathematical and Theoretical},
	44(20):205201, 2011.
	\bibitem{2}
	Antoine Jean-Pierre and Balazs Peter.
	\newblock Frames, semi-frames, and Hilbert scales.
	\newblock {\em Numerical Functional Analysis and Optimization},
	33(7-9):736--769, 2012.
	
	\bibitem{2011}
	Balazs Peter, Diana~T Stoeva, and Jean-Pierre Antoine.
	\newblock Classification of general sequences by frame-related operators.
	\newblock {\em Sampling Theory in Signal and Image Processing}, 10:151--170,
	2011.
	
	\bibitem{2002}
	Casazza Peter G, Christensen Ole, Li Shidong, and Alexander~M Lindner.
	\newblock Riesz-fischer sequences and lower frame bounds.
	\newblock {\em Zeitschrift f{\"u}r Analysis und ihre Anwendungen},
	21(2):305--314, 2002.
	
	\bibitem{7}
	Christensen Ole.
	\newblock {\em An introduction to frames and Riesz bases}, volume~7.
	\newblock Springer, 2003.
	
	\bibitem{Dsum}
	Deguang Han and David~R. Larson.
	\newblock Frames, bases and group representations.
	\newblock {\em Mem. Amer. Math. Soc.}, 147(697):x+94, 2000.
    
         
	\bibitem{kato}
	Kato Tosio.
	\newblock {\em Perturbation theory for linear operators}, volume 132.
	\newblock Springer Science \& Business Media, 2013.
    
	\bibitem{6}
     Koo Yoo Young and Lim Jae~Kun.
	\newblock Existence of {P}arseval oblique duals of a frame sequence.
	\newblock {\em J. Math. Anal. Appl.}, 404(2):470--476, 2013.
    
	
\bibitem{tensor}
     Kubrusly, Carlos S. 
     \newblock A concise introduction to tensor product.
     \newblock{ \em Far East Journal of Mathematical Sciences}, 22(2):137-1-4, 2006.    
    \bibitem{4a}
    Kulkarni, Sudhir H and M Thamban Nair, and  Golla Ramesh
    \newblock Some properties of unbounded operators with closed range.
    \newblock{\em Proceedings Mathematical Sciences, Springer}, 118:613--625, 2008.
  \bibitem{hema}
  Mohanarangan H., P. Sam Johnson, and Harikrishanan P.K.
  \newblock Direct sum of lower semi-frames in Hilbert spaces. 
  \newblock{\em Gulf Journal of Mathematics}, 19(2), 412-423, 2025 
	 \bibitem{infinite}
    Von Neumann, J. 
    \newblock On Infinite Direct Products. 
    \newblock{\em Compositio Mathematica}, 6: 1–77, 1938.
    
	\bibitem{Wiedmann}
     Weidmann, Joachim.
    \newblock Linear operators in Hilbert spaces. Vol. 68. 
    \newblock{Springer Science and Business Media}, 2012.
    \bibitem{tensor dual}
    Wang, Ya-Hui and Yun-Zhang Li.
    \newblock Tensor product dual frames.
    \newblock{ \em Journal of Inequalities and Applications}, 2019(1) :76, 2019.
    
\end{thebibliography}

\end{document}